\documentclass[12pt]{amsart}
\usepackage[margin=1in]{geometry}
\usepackage[english]{babel}
\usepackage[utf8]{inputenc}
\usepackage{subcaption}
\usepackage{amsmath}
\usepackage{amssymb}
\usepackage{amsfonts}
\usepackage{amsthm}
\usepackage{mathrsfs}
\usepackage[all]{xy}
\usepackage[pdftex]{graphicx}
\usepackage{color}
\usepackage{cite}
\usepackage{url}
\usepackage{indent first}
\usepackage[labelfont=bf,labelsep=period,justification=raggedright]{caption}
\usepackage[english]{babel}
\usepackage[utf8]{inputenc}
\usepackage[colorlinks=true, citecolor=cyan, linkcolor=magenta]{hyperref}
\usepackage[colorinlistoftodos]{todonotes}
\usepackage{tkz-fct}
\usepackage{tikz}
\usepackage{color, colortbl}
\usepackage{mathtools}
\usetikzlibrary{calc}

\newcounter{NoTableEntry}
\renewcommand*{\theNoTableEntry}{NTE-\the\value{NoTableEntry}}

\newcommand{\Mod}[1]{\ (\mathrm{mod}\ #1)}

\theoremstyle{plain}
\newtheorem{thm}{Theorem}
\newtheorem{lemma}[thm]{Lemma}
\newtheorem{cor}[thm]{Corollary}

\theoremstyle{definition}
\newtheorem{defn}[thm]{Definition}

\theoremstyle{remark}
\newtheorem{rem}[thm]{Remark}

\numberwithin{equation}{section}
\numberwithin{thm}{section}

\newcommand{\Gal}{{\rm Gal}}

\title[]{Fourier Coefficients of Level 1 Hecke Eigenforms} 
\author{Mitsuki Hanada and Rachana Madhukara}
\address{M. Hanada: Department of Mathematics, Wellesley College, Wellesley, MA 02481}
\email{mhanada@wellesley.edu}
\address{R. Madhukara: Department of Mathematics, Massachusetts Institute of Technology, Cambridge, MA 02139}
\email{rachanam@mit.edu}


\begin{document}

\begin{abstract}
    Lehmer's 1947 conjecture on whether $\tau(n)$ vanishes is still unresolved. In this context, it is natural to consider variants of Lehmer's conjecture. We determine many integers that cannot be values of $\tau(n)$. For example, among the odd numbers $\alpha$ such that $|\alpha|<99$, we determine that $$\tau(n) \notin \{-9,  \pm 15, \pm 21, -25, -27, -33, \pm 35, \pm 45, \pm 49, -55, \pm 63, \pm 77, -81, \pm 91 \}.$$ Moreover, under GRH, we have that $\tau(n) \neq -|\alpha|$ and that $\tau(n) \notin \{9,25,27,39,$ $75,81\}.$ We also consider the level 1 Hecke eigenforms in dimension 1 spaces of cusp forms. For example, for $\Delta E_4 = \sum_{n = 1}^{\infty} \tau_{16}(n)q^n$, we show that \begin{align*}\tau_{16}(n) \notin &\{\pm \ell: 1\leq \ell \leq 99, \ell \text{ is odd}, \ell \neq 33,55,59,67,73,83,89,91\} \\ & \quad \quad \quad \quad \quad \cup \{-33,-55,-59,-67,-89,-91\}.
    \end{align*} Furthermore, we implement congruences given by Swinnerton-Dyer to rule out additional large primes which divide numerators of specific Bernoulli numbers. To obtain these results, we make use of the theory of Lucas sequences, methods for solving high degree Thue equations, Barros' algorithm for solving hyperelliptic equations, and the theory of continued fractions. 
\end{abstract}
\maketitle
\vspace{-8mm}
\section{Introduction and statement of results}
Ramanujan's tau-function $\tau(n)$ is defined to be the coefficients of the weight 12 cusp form $$\Delta(z) = \sum_{n = 1}^{\infty} \tau(n) q^n \coloneqq q \prod_{n = 1}^{\infty}(1-q^n)^{24} = q - 24 q^2 + 252 q^3 -1472 q^4 + \cdots,$$ where $q \coloneqq e^{2\pi i z}$. Over the past 100 years, the tau-function has played an important role in the development of the theory of modular forms. When the tau-function was introduced in his landmark 1916 paper titled ``On certain arithmetical functions," Ramanujan  made several conjectures on the properties of these coefficients. His theories on the multiplicative nature of these coefficients and later work by Mordell and Hecke offered glimpses into the theory of Hecke operators and later paved the way for the Atkin-Lehner theory of newforms. Similarly, bounds that Ramanujan conjectured for these coefficients were proven by Deligne using his famous work on the Weil conjectures.

Although Ramanujan was not able to make much progress on his conjectures, he was able to prove several exceptional congruences for the tau-function:

\vspace{-2mm}

\begin{equation} \label{eq:congs}
    \tau(n) \equiv 
\begin{cases} 
      n^2 \sigma_1(n)  &\pmod 9, \\
      n \sigma_1(n) & \pmod 5, \\
      n \sigma_3(n) & \pmod 7, \\
      \sigma_{11}(n) & \pmod {691},
\end{cases} 
\end{equation}
where $\sigma_v(n) \coloneqq \sum_{d|n} d^v$. Serre recognized these congruences as glimpses of the theory of modular $\ell$-adic Galois representations \cite{serre1967interpretation}. This observation was expanded by Swinnerton-Dyer \cite{swinnerton1973}, who determined similar congruences for level 1 modular forms of weights 16, 18, 20, 22, and 26.  We denote these weight $2k$ forms as $\Delta_{2k}(z)$, and we denote the corresponding coefficients of the Fourier expansions by   
$$\Delta_{2k}(z) = \sum_{n = 1}^{\infty}\tau_{2k}(n) q^n.$$

Despite the extensive theory that it has inspired, some of the tau-function’s most basic properties remain unknown. For example, Lehmer's Conjecture that $\tau(n)$ never vanishes remains open \cite{lehmer1947vanishing}. However, recent work has been focused on possible odd values of newform coefficients. In 1987, Murty, Murty, and Shorey \cite{murty1987odd} proved that for odd $\alpha$, we have $\tau(n) = \alpha$ for at most finitely many $n$. However, their method is computationally ineffective due to the enormous bounds that arise from linear forms in logarithms. Before 2020, the classification of the solutions $n$ of $\tau(n) = \alpha$ had only been carried out for $\alpha = \pm 1$. It is widely believed that for $\alpha = \pm \ell$, where $\ell$ is almost any odd prime, there are no solutions (see \cite{BCOshort, balakrishnan2020variants}). However, there are some large primes $\ell$ for which $\tau(n) = \pm \ell$ has solutions (see \cite{lehmer1965primality, lygeros2013odd}).

More recently, there has been work by Balakrishnan, Craig, Ono, and Tsai  \cite{BCOshort, balakrishnan2020variants} that investigates these questions for even weight newforms with integer coefficients and trivial mod 2 residual Galois representation. They rule out several odd values of $\tau(n)$ by using the theory of primitive prime divisors of Lucas sequences. For $n>1$, they show that $$\tau(n) \notin \{ \pm 1, \pm 3, \pm 5, \pm 7, \pm 13, \pm 17, -19, \pm 23, \pm 37, \pm 691 \}$$ and assuming the Generalized Riemann Hypothesis (GRH) they show that $$\tau(n) \notin \left\{ \pm \ell : 41 \leq \ell \leq 97 \text{ with } \left(\frac{\ell}{5}\right) = -1\right\} ~ \cup ~ \{ -11, -29, -31, -41, -59, -61, -71, -79, -89\}. $$  

We investigate similar questions in the context of level one Hecke eigenforms. We rule out all $-|\alpha|$ such that $|\alpha| < 99$ as coefficients of the $\tau$-function as well as certain other positive odd composite numbers. 

\begin{thm}\label{tauthm}
For every $n>1$, the following are true.
\begin{enumerate}
    \item We have  
    \begin{align*}
        \tau(n) \not\in \{ -9,  \pm 15, \pm 21, -25, -27, -33, \pm 35, \pm 45, \pm 49, -55, \pm 63, \pm 77, -81, \pm 91\}. 
    \end{align*}
    \item Furthermore, assuming GRH, we have  
        \begin{align*} 
        \tau(n) \not\in \{-\ell: 1\leq \ell < 99, \ell \text{ is odd}\} ~ \cup ~ \{9,15,21,25,27,35,39,45,49,63,75,77,81,91\}.
        \end{align*}
\end{enumerate}
\end{thm}

We can extend this theorem to other level 1 Hecke eigenforms $\Delta_{2k}(z)$ and rule out certain odd values as values of $\tau_{2k}(n)$.

\begin{thm} \label{level1thm}
For every $n>1$, the following are true. 

\begin{enumerate}
    \item For $2k = 16$, we have  
    \begin{align*}\tau_{16}(n) \notin &\{\pm \ell: 1\leq \ell \leq 99, \ell \text{ is odd}, \ell \neq 33,55,59,67,73,83,89,91\} \\ & \quad \cup \{-33,-55,-59,-67,-89,-91\}.
    \end{align*} 
    Assuming  GRH, we have \[\tau_{16}(n) \notin \{\pm \ell: 1\leq \ell \leq 99, \ell \text{ is odd}, \ell \neq 33,55,67,91\} \cup  \{-33,-55,-67,-91\}.\] 

    \item For $2k = 18$, assuming GRH, we have \[\tau_{18}(n) \notin \{-\ell: 1\leq \ell \leq 50, \ell \text{ is odd}, \ell \neq 29\}.\]
    
    \item For $2k =20$, assuming GRH, we have \[\tau_{20}(n) \notin \{-\ell: 1\leq \ell \leq 50, \ell \text{ is odd}, \ell \neq 23,29,31,39,41,47\}.\]
    
    \item For $2k = 22$, we have
    \begin{align*}
        \tau_{22}(n) \notin &\{\pm \ell: 1\leq \ell \leq 99, \ell \text{ is odd}, \ell \neq 19,33,39,57,63,69,73,77,83,87,91,93\}  \\
    & \quad \cup  \{-19,-33,-39,-57,-63,-69,-77,-87,-91,-93\}.
    \end{align*}
    Assuming GRH, we have 
    \begin{align*}
        \tau_{22}(n)  \notin &\{\pm \ell: 1\leq \ell \leq 99, \ell \text{ is odd}, \ell \neq 19,33,39,57,87,91,93\}\\ & \quad \cup  \{-19,-33,-39,-57,-87,-91,-93\}.
    \end{align*}
    
    \item For $2k = 26$, 
    we have 
    \begin{align*}
    \tau_{26}(n) \notin &\{\pm \ell: 1\leq \ell \leq 99, \ell \text{ is odd}, \ell \neq 33,55,59,67,73,83,89\} \\ & \quad \cup  \{-33,-55,-59,-67,-89\}.
    \end{align*}
    Assuming GRH, we have \[\tau_{26}(n) \notin \{\pm \ell: 1\leq \ell \leq 99, \ell \text{ is odd}, \ell \neq 33,55,67\} \cup  \{-33,-55,-67\}.\] 
\end{enumerate}
\end{thm}

Among the congruences proven by Swinnerton-Dyer for $\Delta_{2k}(z)$, there are some regarding larger exceptional primes which are divisors of numerators of the Bernoulli numbers $B_{2k}$. For example, in the case of $\tau(n)$ (see (\ref{eq:congs})), this corresponding exceptional prime is 691. For the other $\Delta_{2k}(z)$, the set $(\ell, 2k) \in \{(131,22), (283,20), (593,22), (617,20), (3617,16),$ $(43867,18),$ $(657931,26)\}$ gives us the remaining corresponding exceptional primes and their weights, where $\ell$ is the prime and $2k$ denotes the weight. For $n$ coprime to $\ell$, the following congruence holds:  
\begin{equation}\label{eq:congsbigprime}
\tau_{2k}(n) \equiv \sigma_{2k-1}(n) \pmod \ell. 
\end{equation}
We use these congruences to determine more inadmissible values of $\tau_{2k}(n)$. 

\begin{thm}\label{bigprimes}
The following are true. 
\begin{enumerate}  
\item We have that  $\tau_{20}(n) \neq \pm 617$ and $\tau_{22}(n) \neq \pm 131$.
\item Assuming GRH, we have that $\tau_{22}(n) \neq \pm 593$.
\end{enumerate}
\end{thm}

The methods used to prove Theorem \ref{bigprimes} can be applied to 283, 3617, 43867, and 657931  as well. However, certain technical difficulties arise due to the complexity of the problems involving larger primes. Though we are unable to rule these large primes out as values of $\tau_{20}(n),\tau_{16}(n),\tau_{18}(n),$ and $\tau_{26}(n)$ respectively, we can limit the possible values $n$ of where they might occur. 

\begin{thm}\label{bigprimes2}
The following are true.  

\begin{enumerate}
    \item If $\tau_{20}(n) = \pm 283$, then $n = p^{2}$ for some prime $p$.
    \item If $\tau_{16}(n) = \pm 3617$, then $n = p^{112}$ for some prime $p$.
    \item If $\tau_{18}(n) = \pm 43867$, then either $n = p^{2436}$ or $n= p^2$ for some prime $p$.
    \item If $\tau_{26}(n) = 657931$, then $n = p^{240}$ or $n = p^2$ for some prime $p$. 
    \item If $\tau_{26}(n) = -657931$, then $n = p^{240}$.
\end{enumerate}
\end{thm}

These theorems are obtained from the relation we can draw between inadmissible values of $\tau_{2k}(n)$ and the integer solutions of Thue equations and hyperelliptic curves. In order to utilize this connection, in Sections \ref{lucassequence} and \ref{2kdvals} we prove a series of lemmas that expand upon the results stated in \cite{balakrishnan2020variants} and also others that allow us to constrain possible values of $n$. Specifically, we show that for certain odd composite numbers, $n$ must be a power of a prime by using the properties of the sequence $\{\tau_{2k}(p^m)\}$. 

Furthermore, using the congruences in \cite{swinnerton1973}, we are able to restrict the possible exponents of the primes depending on the prime factors of the odd composite numbers. This theory allows us to reduce the problem of inadmissible coefficients to solving for integer points on specific curves. In particular, the question of whether a certain odd number can be a value of $\tau_{2k}(p^2)$ or $\tau_{2k}(p^4)$ can be reduced to the question of finding integer solutions on corresponding hyperelliptic curves. Similarly, other even exponents correspond to questions about integer points on certain Thue equations. In Section \ref{curves}, we relate the problem of solving $\tau_{2k}(n) = \alpha$ to finding integer points on specific Diophantine equations. We also discuss the explicit methods we use to find these integer solutions. Lastly, in Section \ref{tau} we rule out additional composite values of $\tau(n)$, and in Section \ref{weights} we establish results for other level one cusp forms. 

\section*{Acknowledgements}

The authors would like to thank Professor Ken Ono, Wei-Lun Tsai, Will Craig, and Badri Vishal Pandey for their guidance and suggestions. They would also like to thank the anonymous referee for helpful feedback. This research was generously supported by the National Science Foundation Grant DMS-2002265, National Security Agency Grant H98230-20-1-0012, the Templeton World Charity Foundation, and Thomas Jefferson Fund at the University of Virginia. 

\section{Lucas Sequences}\label{lucassequence}
We employ the work of Bilu, Hanrot, and Voutier in \cite{bilu1999existence} in order to better understand the coefficients of even weight newforms. 

\subsection{Lucas Numbers and Primitive Prime Divisors}

\begin{defn}
A \textit{Lucas pair} is a pair $(\alpha, \beta)$ of algebraic integers such that $\alpha+\beta$ and $\alpha\beta$ are non-zero coprime rational integers and $\alpha/\beta$ is not a root of unity. The associated \textit{Lucas numbers} $\{u_n(\alpha,\beta)\}=\{u_1=1,u_2=\alpha+\beta,\dots \}$ are the integers $$u_n(\alpha, \beta) \coloneqq \frac{\alpha^n-\beta^n}{\alpha-\beta}.$$
\end{defn}

Lucas numbers are classical objects and have been studied extensively in the past by many authors. Here we consider the existence of primitive prime divisors. 

\begin{defn}
Let $(\alpha, \beta)$ be a Lucas pair. A prime number $p$ is a \textit{primitive prime divisor} of $u_n(\alpha, \beta)$ if $p$ divides $u_n$ but does not divide $(\alpha-\beta)^2 u_1(\alpha, \beta) \cdots u_{n-1}(\alpha, \beta)$. Furthermore, if $u_n(\alpha,\beta)$ for $n>2$ does not have a primitive prime divisor, then it is called \textit{defective}.
\end{defn}

\begin{rem}
An integer $n$ is \textit{totally non-defective} if every Lucas pair $(\alpha, \beta)$ has a primitive divisor at $u_n(\alpha, \beta)$. 
\end{rem}

Bilu, Hanrot, and Voutier prove the following definitive theorem. 

\begin{thm}[Theorem 1.4 in \cite{bilu1999existence}]
Every Lucas number $u_n(\alpha, \beta)$, with $n > 30$, has a primitive prime divisor.
\end{thm}

Furthermore, their work, combined with the subsequent work of Abouzaid \cite{abouzaid2006nombres}, gives the \textit{complete classification} of defective Lucas numbers. This classification shows that each defective Lucas number either belongs to a finite list of sporadic examples or a finite list of parameterized infinite families. In our work, we specifically consider Lucas sequences arising from those quadratic integral polynomials $$F(X)=X^2 - AX + B = (X - \alpha)(X - \beta),$$ where $B = \alpha\beta=p^{2k-1}$ is an odd power of a prime, and $|A| = |\alpha+\beta| \leq 2 \sqrt{B} = 2p^{\frac{2k-1}{2}}$. 

\subsection{Lucas Sequence arising from Hecke eigenforms} \label{proplucas}

Throughout this paper we let $$f(z) \coloneqq q + \sum_{n=2}^{\infty}a_f(n)q^n \in S_{2k}(\Gamma_0(N)) \cap \mathbb{Z}[[q]]$$ be an even weight $2k \geq 4$ newform. For basic facts about newforms, the reader may consult \cite{apost0lbook, webofmod, atkin1970hecke}. We have the following theorem alongside a deep theorem of Deligne \cite{deligne1974conjecture, deligne1980conjecture} which gives us properties of the coefficients $\{a_f(p), a_f(p^2), \dots \}$. 

\begin{thm} \label{recurrence}
If $p\nmid N$ is prime then we have the following.
\begin{enumerate}
    \item If $m\geq 2$, then we have that $$a_f(p^m) = a_f(p)a_f(p^{m-1}) -p^{2k-1}a_f(p^{m-2}).$$
    \item If $\alpha_p$ and $\beta_p$ are roots of $F_p(x) = x^2 - a_f(p)x + p^{2k-1}$, then we have that $$a_f(p^m) = u_{m+1}(\alpha_p,\beta_p) =  \frac{\alpha_p^{m+1}-\beta_p^{m+1}}{\alpha_p-\beta_p}.$$ Moreover, we have that $\vert a_f(p)\vert \leq 2p^{\frac{2k-1}{2}}$ and $\alpha_p$ and $\beta_p$ are complex conjugates. 
\end{enumerate}
\end{thm}

It is easy to see with some manipulation that the two statements above are equivalent. Next, we use properties of the Lucas sequence $\{a_f(p), a_f(p^2), \dots \}$ to get the following result.

\begin{lemma}\label{allodd} 
Let $f(z) = q + \Sigma_{n=2}^{\infty}a_f(n)q^n$ be an even weight $2k \geq 4$ newform of level $N$ with integer coefficients with trivial mod $2$ residual Galois representation.  If $p \nmid N$ is an odd prime, then the Lucas sequence $\{ a_f(p^m) \}$ has no odd defective values.
\end{lemma}

\begin{proof}
From the classification of the defective cases mentioned in \cite{balakrishnan2020variants},  we know the only sporadic cases that occur are those for weight $4$. They occur at $(a_f(2), a_f(4)) = (\pm 3, 1)$ and $(a_f(2), a_f(32)) = (\pm 5, \pm 85)$, neither of which are part of the sequence $\{ a_f(p^m) \}$. 

Otherwise, the cases in which defects can occur are parameterized by infinite families which appear in \cite[Table 2]{balakrishnan2020variants}. Let $m$ be a positive even integer, so that $a_f(p^m)$ is odd.  Since $f$ is of weight greater than or equal to four, we only have to consider rows two, three, and six of \cite[Table 2]{balakrishnan2020variants}. Since $f$ has a trivial mod $2$ Galois representation, $a_f(p)$ must be even for odd $p \nmid N$. This implies that $a_f(p) \nmid a_f(p^m)$, thus rows three and six cannot hold. This leaves us with just row two. 

Now we consider row two. If $a_f(p^m) = \pm \varepsilon \cdot 3^r$ is defective where $\varepsilon = \pm 1$, then we know $m = 2$. However, the constraints on the parameters require that  $3 \nmid a_f(p)$, which is a contradiction.
\end{proof}

The following lemma describes the types of odd numbers which must appear in Lucas sequence of the form $\{1,a_f(p), a_f(p^2)\dots \}$. 


 \begin{lemma} \label{inductivestep}
 Let $c$ be an odd integer such that for all pairs $(\alpha, \beta)$ where $c = \alpha \beta$ we have either $|a_f(n)| \neq \alpha$ or $|a_f(n)| \neq \beta$ for any $n\in \mathbb{Z}$. If we have that $a_f(n_0) = c$, then $n_0 = p^d$ where $p$ is prime. 
\end{lemma}


\begin{proof}
Suppose $n_0$ satisfies $a_f(n_0)=c$. If $n_0$ is a composite number with at least two distinct prime factors, then we can write $n_0 = nm$ where $n$ and $m$ are coprime. By Hecke multiplicativity, we know $c = a_f(nm) = a_f(n) a_f(m)$, which implies that $(a_f(n),a_f(m))$ is a pair $(\alpha,\beta)$ such that $\alpha\beta = c$. However, since there exists no pair of coprime integers $n,m$ that satisfy this, it follows that $n_0$ is a power of a prime.
\end{proof}

\subsection{Limiting the potential values of $n$} \label{limitdvals}

Following the notation given in \cite{balakrishnan2020variants}, let \[\mathcal{U}_f \coloneqq \{1\} \cup \{4: \text{ if } a_f(2) = \pm 3, 2k = 4, N \text{ odd}\}.\] Moreover, let $m_\ell(\alpha_p, \beta_p)$ be given by \[m_\ell(\alpha_p, \beta_p) \coloneqq \min\{n \geq 1: a_f(p^{n-1}) \equiv 0 \Mod\ell\}.\] 

\begin{thm}[Theorem 3.2 in \cite{balakrishnan2020variants}] \label{bigtheoremono}
Suppose that the mod $2$ residual Galois representation for $f(z)$ is trivial. If $|a_f(n)| = \ell^m$, with $m \in \mathbb{Z}^+$ and $\ell$ is an odd prime, then $n = m_0 p^{d - 1}$ where $m_0 \in \mathcal{U}_f$, $p \nmid N$ is prime, and $d | \ell(\ell^2 - 1)$ is an odd prime. Moreover, $|a_f(n)| = \ell^m$ for finitely many (if any) $n$. 
\end{thm}
 
\begin{cor} \label{bigcorollary}
Suppose $f(z) = q + \Sigma_{n=2}^{\infty}a_f(n)q^n$ is an even  weight $2k \geq 4$ newform of level $N$. If a Fourier coefficient of this newform is of the form $a_f(n) = \ell_1^{r_1} \dots \ell_s^{r_s}$ for some $n$ where the $\ell_i$ are odd primes such that $\ell_i \nmid N$ and the conditions of Lemma \ref{inductivestep} are satisfied, then $n = p^{d-1}$ where $p \nmid N$ and \[d \in \bigcup_i ~ \{m: m \text{ odd prime, } m|\ell_i(\ell^2_i - 1)\}.\] Moreover, $|a_f(n)| = \ell_1^{r_1} \dots \ell_s^{r_s}$ for finitely many $n$. 
\end{cor}

\begin{proof}
Since $a_f(n)$ satisfies the conditions of Lemma \ref{inductivestep}, we have that $n = p^t$ where $p$ is prime. The $\ell_i$ are odd primes, so $a_f(p^t)$ must be odd. By Lemma \ref{allodd}, $a_f(p^t)$ is a non-defective term in the Lucas sequence $\{1, a_f(p), a_f(p^2), \dots \}.$ By non-defectivity, for at least one prime $\ell_i \in \{\ell_1,\ell_2, \dots \ell_s\}$, the term $a_f(p^t)$ must be the first element of the Lucas sequence such that $\ell_i \vert a_f(p^t)$. Therefore, the value of $t + 1$ must coincide with the value of $m_{\ell_i}(\alpha_p, \beta_p)$ for at least one $\ell_i$. By Theorem \ref{bigtheoremono}, the values of $m_{\ell_i}(\alpha_p, \beta_p)$ for each $\ell_i$ are given by the condition $m_\ell(\alpha_p, \beta_p) | \ell(\ell^2 - 1)$. 

By the Hecke recurrence and the triviality of the Galois representation mod 2, we have that $a_f(p^{t})$ is odd if and only if $t$ is even. It is a classical fact that if $(r+1)|(t+1)$, then $a_f(p^{r})|a_f(p^{t})$. Because $a_f(n)$ satisfies the conditions of Lemma \ref{inductivestep}, there can be no $a_f(p^{r})$ that satisfies this divisibility condition. Thus, $t+1$ must be an odd prime. From the Hecke recurrence relation (see Theorem \ref{recurrence}) and the fact that the weight is $2k \geq 4$, we know that the corresponding curve has positive genus. Therefore, Siegel's Theorem asserts that $|a_f(p^{t})| = c$ has only finitely many integer points.

\end{proof}

\begin{rem}
Let $a_f(n)$ be the Fourier coefficients of a weight $2k = 4$ level $N$ newform that satisfy the conditions given in Lemma \ref{inductivestep}. If $(a_f(2),  a_f(32)) \neq (\pm 5, \pm 85)$ and $a_f(n) = \alpha$ where $\alpha$ is an odd integer for which all the odd prime factors $\ell_i \mid \alpha $ are such that $\ell_i \nmid N$, then $n = m_0 p^{d-1}$ where $d$ and $p$ satisfies the same conditions as before  and $m_0 \in \mathcal{U}_f$. This follows directly from Theorem \ref{bigtheoremono} and Corollary \ref{bigcorollary}. If $(a_f(2), a_f(32)) = (\pm 5, \pm 85)$, then Corollary \ref{bigcorollary} still holds for $\alpha$ where $85 \nmid \alpha$. 
\end{rem}

\section{Limiting potential values of $d$ using congruences} \label{2kdvals}

In the specific case of level one cusp forms that form a one dimensional vector space over $\mathbb{C}$, we utilize the work of Swinnerton-Dyer on congruences. It is well known that all of these forms have residually reducible mod 2 Galois representation \cite[Theorem 1.3]{ono20052}. We want to know whether the Fourier coefficients of these cusp forms attain certain odd values and if so where those values are supported. Applying Corollary \ref{bigcorollary}, if the $n$th Fourier coefficient $\tau_{2k}(n)$ of a weight $2k = 12, 16, 18, 20, 22,$ or $26$ cusp form is odd and satisfies the conditions of Lemma \ref{inductivestep}, then we have that $n = p^{d-1}$ for some odd prime $d$. 

The congruences outlined in the work of Swinnerton-Dyer\cite{swinnerton1973} on the Ramanujan $\tau$-function and more generally for level one Hecke operators, can give us a stronger statement than Corollary \ref{bigcorollary} for finding whether $\tau_{2k}(n) = c$, where $c$ is an odd integer such that $\ell \mid c$ for prime $\ell$ in Table \ref{tab:dvals} and $c$ satisfies the conditions of Lemma \ref{inductivestep}. In particular, it helps by giving a stronger restraint on the possible $d$ values corresponding to some small primes. 

We wish to define a set $D_{2k,\ell_i}$ which contains all possible $d$ values corresponding to the prime $\ell_i$ for weight $2k$ level 1 cusp forms. For specific primes where $D_{2k,\ell_i}$ depends on the congruences and changes depending on weight, we have classified the corresponding $D_{2k,\ell_i}$ in Table \ref{tab:dvals}.
For all other primes $\ell_i$, we define $D_{2k,\ell_i} \coloneqq \{m: m \text{ odd prime, } m|\ell_i(\ell^2_i - 1)\}$, where $D_{2k,\ell_i}$ does not depend on the weight.  

\begingroup
\setlength{\tabcolsep}{5pt} 
\renewcommand{\arraystretch}{1.6}
\begin{center}
\begin{table}[!htb]
\begin{tabular}{|c|c|c||c|c|c|} 
 \hline
 $\ell$ & Weight ($2k$) & $D_{2k, \ell}$  & $\ell$ & Weight ($2k$) & $D_{2k, \ell}$  \\
 \hline \hline
 $5$  & 12,16,18,20,22,26 &  $\{5\}$  &  $131$  & 22 & $\{5, 13, 131\}$\\
 \hline
 $7$ & 12,18,20,26 & $\{7\}$  & $283$ & 20  & $\{3, 47, 283\}$\\
 \hline 
 $7$  & 16, 22 & $\{3,7\}$ & $593$ & 22 & $\{37, 593\}$ \\
 \hline 
 $11$  & 16,20,26 & $\{5, 11\}$ & $617$  & 20 & $\{7, 11, 617\}$\\
 \hline 
 $11$  & 18  & $\{11\}$ & $691$ & 12 & $\{3, 5, 23, 691\}$\\
 \hline 
 $13$  & 18,20,22 & $ \{3,13\}$ & $3617$  & 16 & $\{113, 3617\}$\\
 \hline
 $17$  & 22,26 & $ \{17\}$ & $43867$  & 18 & $\{3, 2437, 43867\}$\\
 \hline 
 $19$  & 26  &  $\{3, 19\}$ & $657931$  & 26 & $\{3, 7, 13, 241, 657931\}$\\
 \hline
\end{tabular}
\medskip
\caption{\textit{List of special cases for $d$ values using congruences.}} 
\label{tab:dvals}
\end{table}
\end{center}
\endgroup
\vspace{-12mm}
For level one cusp forms, we can use the following lemma to limit the possible $d$ values we must consider for composite odd numbers. 

\begin{lemma}\label{lemdvalupdated}
Let us suppose that $\tau_{2k}(p^{d-1}) = \ell_1^{r_1} \dots \ell_s^{r_s}$ where $p$ is prime, $\ell_i$ are odd primes and the conditions of Lemma \ref{inductivestep} are satisfied. We define \[M_{2k} \coloneqq \max_i \min_{i, m_i > 0} D_{2k,\ell_i}\] where $m_i \in D_{2k,\ell_i}$ and \[D_{2k,(\ell_1^{r_1} \dots \ell_s^{r_s})}^{*} \coloneqq \bigcup_{i} ~ (D_{2k,\ell_i} \cap \mathbb{Z}_{\geq M_{2k}}).\] Then $d \in D_{2k,\ell_1^{r_1} \dots \ell_s^{r_s}}^{*}.$

\end{lemma}

\begin{proof}
The proof is similar to that of Corollary \ref{bigcorollary}. For each $\ell_i$, we have that $\ell_i \mid \tau_{2k}(p^{d-1})$. By non-defectivity, it follows that $d\geq m_{\ell_i}$ for all $i$, thus it follows that $d\geq M_{2k}$. 
 
\end{proof}

\begin{rem}
In the cases when no $\ell_i$ are in Table \ref{tab:dvals} , then the values given here are identical to those given in Corollary \ref{bigcorollary}. 
\end{rem}

\section{Hyperelliptic curves and Thue Equations}\label{curves}

The problem of determining whether $a_f(p^{d-1}) = c$ for all possible $(p,d)$ can be reduced  to a problem of determining integer points on a set of finitely many algebraic curves. We know that the elements in the Lucas sequence $\{a_f(p), a_f(p^2), \dots \}$ for $p \nmid N$ satisfy the recurrence relation \[ a_f(p^m) = a_f(p)a_f(p^{m-1}) - p^{2k-1}a_f(p^{m-2}).\] Using this recurrence relation, we can find curves that we want to analyze in order to determine if $a_f(p^{d-1}) = c$ is true.

\subsection{Hyperelliptic curves} \label{barros}

For $d-1 = 2,4$, determining whether  $a_f(p^{d-1}) = c$ holds becomes a problem of finding points on corresponding hyperelliptic curves. The following lemma comes from the recursive relation of the Lucas sequence mentioned in the previous section. 

\begin{lemma}\label{lemcurve}
For coefficients $a_f(n)$ of a weight $2k$ level $N$ newform, we have the following for $p$ prime such that $p\nmid  N$. 
\begin{enumerate}
    \item If $a_f(p^2) = \pm c$, then $(p,a_f(p)^2)$ must be a point on the hyperelliptic curve
    
\begin{equation}
    C_{k,c}^{\pm} : Y^2 = X^{2k-1} \pm c.
\end{equation}

\item If $a_f(p^4) = \pm c$, then $(p, 2a_f(p)^2-3p^{2k-1})$ must be a point on the hyperelliptic curve 

\begin{equation}
   H_{k,c}^{\pm} : Y^2 = 5X^{2(2k-1)} \pm 4c.
\end{equation}

\end{enumerate}
\end{lemma}

In order to find the integer points on the hyperelliptic curves, we implement the algorithm outlined by Barros in \cite{barros}. In his thesis, Barros states theorems that give us integer solutions to equations of the form $x^2 + D = Cy^n$ with nonzero integers $C,D$. This method involves producing a finite set of Thue equations and solving for integer points on the Thue equations, which then gives us potential $x$ coordinates.

First, since $D$ is a nonzero integer, there exist nonzero integers $d$ and $q$ where $d$ is square free, $q\geq 1$ and $D = dq^2$. In our case, we will be considering two different situations: when $\sqrt{-d}\in \mathbb{Q}$ and $\sqrt{-d} \not \in \mathbb{Q}$. Note that $\sqrt{-d}\in \mathbb{Q}$ only occurs when $d = -1$, or when $D$ is a negative square. In this case, it follows that our field of interest $K = \mathbb{Q}(\sqrt{-d})$ is just $\mathbb{Q}$ and the ring of integers associated with $K = \mathbb{Q}$ is just $\mathbb{Z}$. Now let us assume that $R = \underset{\text{prime }p \mid 2q}{\prod} p$ denotes the product of all distinct primes that divide $2q$. Then, by \cite[Theorem~2.2]{barros}, we have that for any integer solution $(x,y)$ (where $y\neq 0$) of the equation $x^2 + D = Cy^n$, there exist corresponding natural numbers $a,b,c_1,c_2$ that have the following properties 
\begin{enumerate}
    \item  $c_1c_2 = C$
    \item $a \mid R^n$ and $b\mid R^n$
    \item for prime $p$, we have $p\mid a $ if and only if $p \mid b$
    \item $ab$ is a perfect $n$th power.
\end{enumerate}
and give us the following corresponding Thue equation \[2q = bc_2U^n-ac_1V^n.\]
Then a solution  $(U,V)$ of the above Thue equation gives us that \[x = \frac{1}{2}(bc_2U^n + ac_1V^n).\] Therefore, we can use this property to find all possible values of $x$, by first finding all possible combinations of $(a,b,c_1,c_2)$ and then generating the corresponding Thue equations. We can find solutions $(U,V)$ of said Thue equations and use them to get explicit equations that give us potential values of $x$, or potential integer solutions of the equation $x^2 + D = Cy^n$.

When $D$ is not a negative square, it follows that $\sqrt{-d}\not \in \mathbb{Q}$. Let $K = \mathbb{Q}(\sqrt{-d})$ and $\mathcal{O}_K$ be the ring of integers associated with $K$. Let $\sigma\in \Gal(K/\mathbb{Q})$ such that $\sigma(\sqrt{-d}) = -\sqrt{-d}$. For elements $\alpha \in K$, let $\bar{\alpha}$ denote the conjugate of $\alpha$, which is defined by $\bar{\alpha} = \sigma(\alpha)$. 
For any integer solution $(x,y)$  (where we have that $y\neq 0$) of $x^2 + D = Cy^n$, \cite[Theorem~2.1]{barros} states that there exists a finite set $\Gamma$ with pairs $(\gamma_{+},\gamma_{-})$ of elements in $K$ that gives us the Thue equation\[2q = \frac{1}{\sqrt{-d}}(\gamma_{+}(A+B\omega)^n - \gamma_{-}(A+B\Bar{\omega})^n),\] where $\{1,\omega\}$ is the integral basis of $\mathcal{O}_K$ and $\bar{\omega}$ is the conjugate of $\omega$. A solution $(A,B)$ of this Thue equation gives us an explicit formula for $x:$\[x = \frac{1}{2}(\gamma_{+}(A+B\omega)^n + \gamma_{-}(A+B\Bar{\omega})^n).\]We can use this property to find all possible integer solutions $(x,y)$ of $x^2 + D = Cy^n$ by finding all the elements of $\Gamma$ and then generating a finite set of Thue equations. The solutions $(A,B)$ of these Thue equations can then be used to plug into the explicit formula that will produce potential values of $x$.

In order to construct $\Gamma$, we will construct potential  ideals $\mathfrak{a}_+$ of the form \[ \mathfrak{a}_+ = \underset{\mathfrak{p} \mid \langle 2DC\rangle}{\prod} \mathfrak{p} ^{\kappa_\mathfrak{p}},\] where $\mathfrak{p}$ are prime ideals and $0\leq \kappa_{\mathfrak{p}} < n$ such that \begin{enumerate}
    \item $0\leq \kappa_\mathfrak{p}\leq \kappa_1$ where $\kappa_1$ is the largest  integer such that $\mathfrak{p}^{\kappa_1} \mid \langle 2qC\sqrt{-d} \rangle $ holds,
    \item min $\{\kappa_\mathfrak{p}, \kappa_{\bar{\mathfrak{p}}} \} \leq \kappa_2 $ where $\bar{\mathfrak{p}}$ and $\kappa_2$ is the largest  integer such that $\mathfrak{p}^{\kappa_2} \mid \langle 2q\sqrt{-d}\rangle  $ holds,
    \item $\kappa_\mathfrak{p} + \kappa_{\bar{\mathfrak{p}}} - \kappa_3 \equiv 0  \text{ (mod } n)$ where $\kappa_3$ is the largest integer such that $\mathfrak{p}^{\kappa_3} \mid \langle C\rangle $.
\end{enumerate}
We find all possible ideals $\mathfrak{a}_+$ by first finding all prime ideals $\mathfrak{p}$ that divide $\langle 2DC \rangle $ and and their corresponding possible  $\kappa_{\mathfrak{p}}$. We then take the combinations of these primes ideals to get all the possible products of the form desired. Next, we find the integral ideals $\mathfrak{g}_1,...\mathfrak{g}_h$ of $\mathcal{O}_K$ (where $h$ is the order of the class group of $\mathcal{O}_K$). We combine the integral ideals and the potential $\mathfrak{a}_+$ to find cases where $\mathfrak{a}_+ \mathfrak{g}^{-n}_i$ is a principal ideal. In the cases where  $\mathfrak{a}_+ \mathfrak{g}^{-n}_i$ is indeed a principal ideal, we find a generator of the ideal, $\gamma'$. From this, we can construct  the set $\Gamma'$, which is a collection of pairs $(\gamma', \bar{\gamma'})$ of the generators $\gamma'$ and their conjugate. 

Now we will construct the set $\Gamma$ using the set $\Gamma'$ and the representatives of the units of $K$ modulo $n$th powers, $U_K/U_K^n$. For each pair $(\gamma', \bar{\gamma'})\in \Gamma'$ and representative $u\in U_K/U_K^n$, we can construct a new pair $(\gamma_u,\bar{\gamma_u})$. Let $\Gamma$ be the collection of such pairs $(\gamma, \bar{\gamma}) = (\gamma_u,\bar{\gamma_u})$ for some $u \in U_K/U_K^n$ and some $(\gamma',\bar{\gamma'})\in \Gamma'$. We can use this set $\Gamma$ to find the corresponding Thue equations, and then use integer solutions $(A,B)$ of the Thue equations to get explicit equations that gives us potential values of $x$. 

\subsection{Thue Equations}
For $d-1>4$, we consider integer solutions of Thue equations. Consider the homogeneous polynomials generated by the following power series in $T$:
\begin{equation}\label{genfunction}
\frac{1}{1-\sqrt{Y}T+XT^2}=\sum_{m=0}^{\infty}F_m(X,Y)\cdot T^{m}=1+\sqrt{Y}\cdot T+(Y-X)T^2+\cdots.
\end{equation}  
We have an explicit formula for polynomials $F_{2m}(X,Y)$ of degree $m$:\begin{equation}\label{ThueEqn}
F_{2m}(X,Y)=\prod_{k=1}^m \left(Y-4X \cos^2\left(\frac{\pi k}{2m+1}\right)\right).
\end{equation}
The following lemma relates these polynomials to coefficients of newforms. 

\begin{lemma}\label{lemthueq}
 If $a_f(p^{d-1}) = \pm c$, then $(p^{2k-1},a_f(p)^2)$ must be an integer solution of the Thue equation $F_{d-1}(X,Y) = \pm c$.
\end{lemma}

The lemma above comes from the fact that the sequence of integers  $F_m(p^{2k-1},a_f(p)^2)$  and the Lucas sequence $\{a_f(p^m)\}$ satisfy the same recurrence relation. Using ThueSolver in \texttt{SageMath}, we can find integer solutions for these Thue equations independent of GRH up to $(d-1) \leq 30$. However, Thue equations of higher degree are more difficult to solve. Therefore, we use different methods to simplify the problem of checking whether $F_{\ell-1}(X,Y) = \pm \ell$ has any integer solutions for odd prime $\ell$. 

\subsection{The Method of Continued Fraction} \label{contfrac}
When considering the cases $F_{\ell-1}(X,Y) = \pm \ell $ when $\ell$ is an odd prime, it is equivalent to think about the equation  \[\widehat{F}_{\ell}(X,Y)= \prod_{k=1}^{\frac{\ell-1}{2}} \left(Y-2X \cos\left(\frac{2\pi k}{\ell}\right)\right) = \pm \ell.\] From the explicit formula (\ref{ThueEqn}) for the Thue equations $F_{2m}(X,Y)$, we can easily see that for odd primes $\ell$ we have $F_{\ell-1}(X,Y) = \widehat{F}_{\ell}(X,Y-2X)$. Hence we can alternatively solve for integer points on $\widehat{F}_\ell(X,Y) = \pm \ell$ and break the problem down into three different cases. 
For large $\vert X \vert $, we know there is no solution from by \cite[Corollary~2.5]{bilu1999existence}.

\begin{defn}
Let us define $P'(n)$ in the following way. 
\begin{equation*}
P'(n) = 
\begin{cases}
2 & \textit{if $n = 2^k\cdot 3$ for some $k$}  \\
\text{greatest prime divisor of $n$}  & \textit{otherwise}.
\end{cases}
\end{equation*}
\end{defn}

\begin{cor}[Corollary 2.5 in \cite{bilu1999existence}]
An integer $n$ is totally non-defective if the equation $$\widehat{F}_n(X, Y) \in \{ \pm 1, \pm P'(n) \}$$ has no solutions $(X, Y) \in \mathbb{Z}^2$ with $|X| > e^8$.  
\end{cor}

Thus, we have that $\widehat{F}_{\ell}(X,Y) = \pm \ell$ has no integer solutions $(X,Y)$ when $\vert X \vert > e^8$. For midsize $\vert X \vert$, this equation becomes a problem regarding continued fraction expansions. This leaves a finite set of small $\vert X \vert $ values, which we can check. Definitions of ``large" and ``midsize" depend on the value of $p$ being considered. We know that for primes $\ell \geq 31$, we consider $\vert X \vert > e^8$ to be ``large" values of $\vert X \vert$ as seen in \cite{bilu1999existence}.  The lower bound of the midsize definition depends on $\ell$, as given in \cite[Lemma~1.1]{tzanakis1989practical}. For the cases we discuss, we can show that this lower bound is always $3$ by using the arithmetic of the cyclotomic field. For the midsize region, dividing out $\widehat{F}_{\ell}(X,Y) = \pm \ell$ by $X^{\frac{\ell-1}{2}}$ gives us \[\prod_{k=1}^{\frac{\ell-1}{2}} \left(\frac{Y}{X}-2 \cos\left(\frac{2\pi k}{\ell}\right)\right) = \frac{\pm \ell}{X^{\frac{\ell-1}{2}}} \approx 0.\] In this case, since $\pm \ell/X^{\frac{\ell-1}{2}} \approx 0$,  this equation  holds when $\frac{Y}{X} \approx 2\cos(\frac{2\pi k}{\ell})$ for some $k$. We can find values $(X,Y)$ such that this holds by finding  the continued fraction expansion of $2\cos(\frac{2\pi k}{\ell})$. For each possible $k$, we can find a sequence of pairs $\{(q_{ki},p_{ki})\}$  where $\frac{p_{ki}}{q_{ki}}$ is the $i$th term of the sequence of continued fraction expansion of $2\cos(\frac{2\pi k}{\ell})$. We can check whether $\widehat{F}_{\ell}(q_{ki},p_{ki}) = \pm \ell$ holds for all pairs  $(q_{ki},p_{ki})$ such that $q_{ki}<e^8$.

\section{The $\tau$-function} \label{tau}
The following table neatly summarizes what is previously known about the inadmissible values of $\tau(n)$. We aim to determine what happens at the bold cells of the table. As we can see, the values that are bold in the table are composite. 
\begingroup
\setlength{\tabcolsep}{5pt} 
\renewcommand{\arraystretch}{1.6}
\begin{table}[!htb]
\begin{center}
    \begin{tabular}{|c|c|c|c|c|c|c|c|c|c|}
        \hline
        $\pm 1$ & $\pm 3$ & $\pm 5$ & $\pm 7$ & $\textbf{9}$ & $-11_*$ & $\pm 13$ & $\textbf{15}$ & $\pm 17$ & $-19$   \\
        \hline
        $\textbf{21}$ & $\pm 23$ & $\textbf{25}$ & $\textbf{27}$ & $-29_*$ & $-31_*$ & $\textbf{33}$ & $\textbf{35}$ & $\pm 37$ & $\textbf{39}$  \\ 
        \hline
        $-41_*$ & $\pm 43_*$ & $\textbf{45}$ & $\pm 47_*$ & $\textbf{49}$ & $\textbf{51}$ & $\pm 53_*$ & $\textbf{55}$ & $\textbf{57}$ & $-59_*$  \\
        \hline 
        $-61_*$ & $\textbf{63}$ & $\textbf{65}$ & $\pm 67_*$ & $\textbf{69}$ & $-71_*$ & $\pm 73_*$ & $\textbf{75}$ & $\textbf{77}$ & $-79_*$  \\
        \hline 
        $\textbf{81}$ & $\pm 83_*$ & $\textbf{85}$ & $\textbf{87}$ & $-89_*$ & $\textbf{91}$ & $\textbf{93}$ & $\textbf{95}$ & $\pm 97_*$ & $\textbf{99}$ \\
        \hline
    \end{tabular}
    \medskip
    \caption{\textit{The previously known inadmissible values of $\tau(n)$, where the bold cells denote the plus or minus values that are still unknown. (note. GRH assumption indicated by $_*$)}}
\end{center}
\end{table}
\endgroup

\vspace{-7mm}
In order to determine whether these bold values are inadmissible values of the $\tau$-function, we proceed by using the work done in Section \ref{2kdvals} to determine which curves we need to evaluate. We can use the methods described in previous sections to determine whether the corresponding equations have integer solutions, which help us rule out values of $\tau(n)$. 

\begin{lemma}
For $n>1$, we have that $\tau(n) \neq \pm 15$. 
\end{lemma}

\begin{proof}
We follow the notation from Lemma \ref{lemdvalupdated}. Since min $(D_{12,3}) = 3$ and min $(D_{12,5}) =5$, it follows that $M_{12} = 5$. This implies that $D^{*}_{12,(3 \cdot 5)} = \{5\}$, meaning we only have to consider the case when $d-1 = 4$. Checking whether $\tau(p^4) = \pm 15$ holds corresponds to checking for integer points on the hyperelliptic equation $Y^2 = 5X^{22} \pm 4(15)$. However, since the left hand side is a square, $X^{22} \pm 12$ must be divisible by $5$. This cannot occur since a square modulo 5 is never $\pm 2$, therefore implying that $\tau(n) \neq 15$. 
\end{proof}

Using the methods in Section \ref{curves} to evaluate the corresponding curves, we can get the following lemmas. 
\begin{lemma}\label{lemtaucurvesch}
The following are true. 
\begin{enumerate}  
    \item There are no integer points on the hyperelliptic curves $  C_{6,c}^{\pm}$ for $c\in \{ -9,  -27, -33, \\   -81,-99\}$. 
    
    \bigskip 
    
    \noindent
    Assuming GRH, there are no integer points on the hyperelliptic curves $  C_{6,c}^{\pm}$ for $c\in \{\pm 9, \pm 27, -33, \pm 39, -57, -69, \pm 81, -87, -93, -95\}$.

\bigskip

    \item There are no integer points on the hyperelliptic curves $H_{6,c}^{\pm}$ for $c \in \{\pm 15, -25,  -33,\\  \pm 45, -55,\}. $
   
    \bigskip
   
    \noindent
    Assuming GRH, there are no integer points on the hyperelliptic curves $H_{6,c}^{\pm}$ for $c \in \{\pm 15, \pm 25,  -33,  \pm 45, -55,-57, -65, \pm 75, -87, -93, -95,-99\}.$ 
\end{enumerate}
\end{lemma}

We get these results by solving for integer points on each curve using Barros' algorithm mentioned in Section \ref{barros} and the \texttt{SageMath} code found \cite{hmcode}.

\begin{lemma}\label{lemtauthue}
For $c \in \{21,33,35,39,49,55,57,63,65,69,77,87,91,95,99\}$ and their corresponding $d$ values found using Lemma \ref{lemdvalupdated}, the Thue equations $F_{d-1}(X,Y) = \pm c$ have no integer solutions of the form $(p^{11}, \tau(p)^2)$ where $p$ is prime. 

\end{lemma}

\begin{proof}
The only equation above that has integer solutions is $F_6(X,Y) = \pm 91$. However, none of these solutions have the $x$ coordinate as an 11th powers of a prime.
\end{proof}



\begin{proof}[Proof of Theorem \ref{tauthm}]
For $c$ listed in Theorem \ref{tauthm}, we can find the corresponding $d$ values using Lemma \ref{lemdvalupdated}. From Lemmas \ref{lemtaucurvesch} and \ref{lemtauthue}, we know  that none of these equations have integer solutions of the desired form stated in Lemmas \ref{lemcurve} and \ref{lemthueq}. Thus, it follows that for the corresponding $d$ values, we must have that $\tau(p^{d-1}) \neq c$.
\end{proof}

\section{Level one cusp forms of weight 16,18,20,22, and 26} \label{weights}
We can use similar techniques to expand our results to general $\tau_{2k}(n)$. 

\begin{lemma} \label{level1ch}
The following are true. 
\begin{enumerate}
    \item For the curves of the form $C_{k, \ell}^+$ where $k \in \{ 8, 11, 13 \}$ and $\ell<100$ is odd (as seen in Theorem \ref{level1thm}), we  have integer solutions when $\ell \in \{ 3, 9, 17, 37, 49, 63, 81, 99\}$. These integer solutions are all listed in Table \ref{tab:ccurves} in the Appendix. 
    \item For the curves of the form $H_{k, \ell}^+$ where $k \in \{ 8, 11, 13 \}$ and $\ell<100$ is odd (as seen in Theorem \ref{level1thm}), we  have integer solutions when $\ell \in \{ 5, 11, 19, 25, 29, 41, 55, 71, 89\}$. These integer solutions are all listed in Table \ref{tab:hcurves} in the Appendix. 
\end{enumerate}
\end{lemma}

\begin{rem}
We were not able to determine integer points for the following curves. 
\begin{enumerate}
    \item When $k = 8$, we were not able to determine if $C_{8, 67}^+$, $C_{8, 91}^+$, $H_{8, 33}^+$, or $H_{8, 55}^+$ has integer solutions. 
    \item When $k = 11$, we were not able to determine if $C_{11, 19}^+$, $C_{11, 33}^+$, $C_{11, 39}^+$, $C_{11, 57}^+$, $C_{11, 87}^+$, $C_{11, 91}^+$, or $C_{11, 93}^+$ has integer solutions. 
    \item When $k = 13$, we were not able to determine if $C_{13, 67}^+$, $H_{13, 33}^+$, or $H_{13, 55}^+$ has integer solutions. 
\end{enumerate}
\end{rem}

For the corresponding Thue equations that we want to check, we have the following lemma. 

\begin{lemma}
For odd $c$ where $1\leq c \leq 99$ and $k\in \{8,9,10,11,13\}$, we can determine their corresponding $d$ values using Lemma \ref{lemdvalupdated}. For these $d$ values, all of the Thue equations of the form $F_{d-1}(X,Y) = \pm c$ have no integer solutions of the form $(p^{2k-1}, \tau_{2k}(p)^2)$ for $p$ prime. 
\end{lemma}

\begin{rem}
In the cases where $F_{36}(X,Y) = \pm 73$ or $F_{40}(X,Y) = \pm 83$, we require GRH to conclude that we have no integer solutions of the desired form.
\end{rem}

\begin{lemma} Assuming GRH, for $k \in \{9,10\}$ and $c \in  \{\ell : 1 \leq \ell \leq 50, \text{ $\ell$ is odd}\}$, the following are true. 

\begin{enumerate}
\item Apart from the case when $2k = 20$ and $c\in \{23,31,39,47\}$, we have that if $3 \in D_{2k,c}$, then $C_{k,c}^{-}$ has no integer solutions.
\item Apart from the case when $(2k,c) \in \{(18,29), (20,29) ,(20,31),(20,41)\}$, we have that if $5 \in D_{2k,c}$, then $H_{k,c}^{-}$ has no integer solutions.
\end{enumerate}
\end{lemma}
We can use these results to prove Theorem \ref{level1thm}.

\begin{proof}[Proof of Theorem \ref{level1thm}]
For $c$ listed in Theorem \ref{level1thm}, we can find the corresponding $d$ values using Lemma \ref{lemdvalupdated}. Using these $d$ values, we can check for integer points on the corresponding curves for each $c$ using the \texttt{SageMath} code in \cite{hmcode}. The only ones with integer points are those mentioned in Lemma \ref{level1ch}. We can evaluate each of these integer points to see that none of them satisfy the form listed in Lemmas \ref{lemcurve} and \ref{lemthueq}. 
\end{proof}

For the exceptional large primes listed in Table \ref{tab:dvals}, we can evaluate the corresponding $d$ values to see if they are values of their corresponding $\tau_{2k}(n)$. First, for each $\ell$, we can check the Thue equation of largest degree $F_{\ell-1}(X,Y) = \pm \ell$.  

\begin{lemma} \label{listvals}
For $\ell \in \{131,283,593,617,3617\}$, there exist no integer solutions to the equation \[F_{\ell-1}(X,Y) = \pm \ell\] other than the solution $(\pm 1, \pm 4)$,  
\end{lemma}

\begin{proof}
We use the method mentioned in Section \ref{contfrac} and the \texttt{SageMath} code in \cite{hmcode}. It is well known that there exists no solutions $(X,Y)$ such that $\vert X \vert > e^8$. We then use the continued fraction algorithm to confirm that there are no integer solutions in the midsize range from 3 to $e^8$. Finally, we can check the smaller values when $\vert X\vert \leq 3$ by direct computation. The only solution we can find is $\vert X\vert = 1$.
\end{proof}

Though we were unable to apply the continued fraction method for the larger primes 43867 and 657931 due to their size, we can still derive restrictions on integer solutions of the form $\vert X \vert = p^{2k-1}$. 

\begin{lemma} \label{bigvals}
For $\ell \in \{43867, 657931\}$, there are no primes $p$ for which $(p^{2k-1}, \tau_{2k}(p)^2)$ is an integer solution of \[F_{\ell-1}(X,Y) = \pm \ell.\]
\end{lemma}

\begin{proof}
Since we know that there are no solutions in the range $\vert X\vert > e^8$, it follows that $\vert X\vert \leq e^8 $. Then, if $X =p^{2k-1}$, we have that \[\vert p \vert \leq e^{\frac{8}{2k-1}}, \] which leaves the only possible value of $\vert p\vert $ to be 1, which does not hold.

\end{proof}

Though this gives us intuition on whether these larger primes are inadmissible coefficients, we still need to check the remaining curves corresponding to the $d$ values listed  in Table \ref{tab:dvals}.

\begin{lemma} \label{thm3vals}
For $\ell \in \{131,283,593,617,3617\}$ and the corresponding $d$ values and weights $2k$ listed in Table \ref{tab:dvals}, we have that 
\begin{enumerate}
    \item The equations $H_{11,131}^{-}, F_{12}(X,Y) =-131, F_{6}(X,Y) = \pm 617, \pm 657931$, $F_{10}(X,Y) = \pm 617$, $F_{12}(X,Y) = 657931,$ and $C_{13,657931}^{-}$ have no integer solutions. 
    \item The equation $H_{11,131}^{+}$ has integer solutions $(1,\pm 23)$ and $F_{12}(X,Y) = -131$ has integer solutions $(3,4)$ and $(-3,-4)$. 
    \item Assuming GRH, we have that $F_{46}(X,Y) = \pm 283$ and $F_{36}(X,Y) = \pm 593$ have no integer solutions.
\end{enumerate}
\end{lemma}
We can now use these results to prove Theorems \ref{bigprimes} and \ref{bigprimes2}.

\begin{proof}[Proof of Theorems \ref{bigprimes} and \ref{bigprimes2}]
We note that the values $(1,\pm 23),(3,4),$ and $(-3,-4)$ from Lemmas \ref{listvals} and \ref{thm3vals} do not have the desired form stated in Lemma \ref{lemcurve}. From this, we can conclude that $\pm 131,\pm 617$ are inadmissible values of $\tau_{22}(n),\tau_{20}(n)$ respectively. Furthermore, we can expand this list to include $\pm 593$ for $\tau_{22}(n)$ by using GRH, which proves Theorem \ref{bigprimes}.

From Lemmas \ref{listvals}, \ref{bigvals}, and \ref{thm3vals}, it remains to solve the equations $C_{10,283}^{\pm}, C_{9,43867}^{\pm},$ $C_{13,657931}^{+},$ $F_{112}(X,Y) = \pm 3617, F_{2436}(X,Y)= \pm 43867,$ and $F_{240}(X,Y) = \pm 657931$ in order to conclude that $\pm 283,\pm 3617,$ $\pm 43867, \pm 657931$ are inadmissible values for their respective $\tau_{2k}(n)$. Theorem \ref{bigprimes2} follows directly from this. 
\end{proof}

\section{Appendix}

\begingroup
\setlength{\tabcolsep}{5pt} 
\renewcommand{\arraystretch}{1.6}
\begin{center}
\begin{table}[!htb]
\begin{tabular}{|c||c|c|c|c|c|c|c|c|} 
 \hline
 $\ell$ & 3 & 9 & 17 & 37 & 49 & 63 & 81 & 99 \\
 \hline
 $C_{8, \ell}^+$ & $(1, \pm 2)$ & $(0, \pm 3)$ & $(-1, \pm 4)$ & $(-1, \pm 6)$ & $(0,\pm 7)$ & $(1, \pm 8)$ & $(0,\pm 9)$ & ? \\
 & & & & $(3, \pm 3788)$ & & & & \\
 \hline
 $C_{11, \ell}^+$ & $(1, \pm 2)$ & $(0,\pm 3)$ & ? & $(-1,\pm 6)$ & $(0,\pm 7)$ & $(1, \pm 8)$ & $(0,\pm 9)$ & $(1,\pm 10)$ \\
 \hline
 $C_{13, \ell}^+$ & $(1, \pm 2)$ & $(0, \pm 3)$ & ? & $(-1, \pm 6)$ & ? & $(1, \pm 8)$ & $(0,\pm 9)$ & ? \\ 
 \hline
\end{tabular}
\medskip
\caption{\textit{Integer points on $C_{d, \ell}^+$ \\ (note. GRH assumption indicated by $_*$)}}
\label{tab:ccurves}
\end{table}
\end{center}
\endgroup

\begingroup
\setlength{\tabcolsep}{5pt} 
\renewcommand{\arraystretch}{1.6}
\begin{center}
\begin{table}[!htb]
\begin{tabular}{|c||c|c|c|c|c|c|c|c|c|} 
 \hline
 $\ell$ & 5 & 11 & 19 & 25 & 29 & 41 & 55 & 71 & 89 \\
 \hline
 $H_{8, \ell}^+$ & $(1, 5)$ & $(1, 7)$ & $(1, 9)$ & $(0, 10)$ & $(1, 11)$ & $(1, 13)$ & ? & $(1, 17)$ & $(1, 19)_*$ \\
 \hline
 $H_{11, \ell}^+$ & $(1, 5)$ & $(1, 7)$ & $(1, 9)$ & $(0, 10)$ & $(1, 11)$ & $(1, 13)$ & $(1, 15)$ & $(1, 17)$ & $(1, 19)$ \\
 \hline
 $H_{13, \ell}^+$ & $(1, 5)$ & $(1, 7)$ & $(1, 9)$ & $(0, 10)$ & $(1, 11)$ & $(1, 13)$ & ? & $(1, 17)$ & $(1, 19)_*$ \\ 
 \hline
\end{tabular}
\medskip
\caption{\textit{Integer points $(|X|, |Y|)$ on $H_{d, \ell}^+$ \\ (note. GRH assumption indicated by $_*$)}}
\label{tab:hcurves}
\end{table}
\end{center}
\endgroup

\newpage

\nocite{*}
\bibliographystyle{abbrv}
\bibliography{biblio.bib}

\end{document}